\documentclass[12pt]{amsart}
\usepackage{amsmath}
\usepackage{amsthm}
\usepackage{amssymb}
\usepackage{enumerate}
\usepackage{xcolor}
\usepackage{comment}
\newtheorem{theorem}{Theorem}[section]
\newtheorem{lemma}[theorem]{Lemma}
\newtheorem{proposition}[theorem]{Proposition}
\newtheorem{corollary}[theorem]{Corollary}

\newtheorem{remar}[theorem]{Remark}
\theoremstyle{definition}

\newtheorem{prob}[theorem]{Open Problem}

\newtheorem{Theorem}{Theorem}[section]
\newtheorem{Lemma}[Theorem]{Lemma}
\newtheorem{Corollary}[Theorem]{Corollary}

\newtheorem{Remark}[Theorem]{Remark}

\newcommand{\R}{\mathbb R}

\newcommand{\N}{\mathbb N}
\newcommand{\Z}{\mathbb Z}

\begin{document}

\title{Epsilon multiplicity is a limit of Amao multiplicities}
\author{Steven Dale Cutkosky }
\thanks{The first author was partially supported by NSF}

\address{Steven Dale Cutkosky, Department of Mathematics,
University of Missouri, Columbia, MO 65211, USA}
\email{cutkoskys@missouri.edu}

\author{Stephen Landsittel }

\address{Stephen Landsittel, Department of Mathematics,
University of Missouri, Columbia, MO 65211, USA}
\email{sdlg6f@missouri.edu}


\begin{abstract} In this paper we prove a volume = multiplicity type theorem for epsilon multiplicity. We show that epsilon multiplicity is a limit of Amao multiplicities.
\end{abstract}

\dedicatory{Dedicated to Sudhir Ghorpade  on the occasion of his sixtieth birthday}

\maketitle

\section{Introduction} Let $R$ be a $d$-dimensional Noetherian local ring with maximal ideal $m_R$ and $I\subset R$ be an ideal. The epsilon multiplicity of $I$ is defined in \cite{UV} as 
$$
\epsilon(I)=
\limsup_{n\rightarrow\infty} \frac{\ell_R((I^n)^{\rm sat}/I^n)}{n^d/d!}.
$$
The saturation $J^{\rm sat}$ of an ideal $J$ of $R$ is 
$$
J^{\rm sat}=J:m_R^{\infty}=\cup_{i=0}^{\infty}J:m_R^i
$$
where $m_R$ is the maximal ideal of $R$.
It is shown in Corollary 6.3 \cite{C2} that if $R$ is analytically unramified, then the epsilon multiplicity of an ideal in $R$ exists as a limit; that is, 
$$
\epsilon(I)=
\lim_{n\rightarrow\infty} \frac{\ell_R((I^n)^{\rm sat}/I^n)}{n^d/d!}.
$$
However, the epsilon multiplicity can be an irrational number, as shown in \cite{CHST}.

Some recent papers on epsilon multiplicity are Suprajo Das \cite{sD}, Cutkosky and Sarkar \cite{CS}, Suprajo Das, S. Dubey, S. Roy and J. Verma \cite{DDRV} and Suprajo Das, S. Roy and V. Trivedi \cite{DRT}.

In this paper, we prove that a generalization of  the volume equals multiplicity formula of ordinary multiplicity holds for epsilon multiplicity. The volume = multiplicity formula of ordinary multiplicity is  proven in increasing generality in Ein, Lazarsfeld and Smith \cite{ELS}, Musta\c{t}\u{a} \cite{M1}, Lazarsfeld and Musta\c{t}\u{a} \cite{LM} and  Cutkosky, Theorem 6.5 \cite{C2}. Recently, volume = multiplicity formulas have been proven for $p$-families and weakly graded families, by Sudipta Das \cite{SD}, Sudipta Das and Cheng Meng \cite{DM} and Sarkar \cite{PS}.

Before stating our theorem, we begin by recalling a few facts about   Amao multiplicity, which is defined and developed in \cite{Am}, \cite{R} and on page 332 of \cite{HS}. Suppose that $R$ is a $d$-dimensional Noetherian local ring and $J\subset I$ are ideals in $R$ such that $I/J$ has finite length. Then $I^n/J^n$ has finite length for all $n\ge 0$ and there exists a numerical polynomial $p(n)$ such that for $n$ sufficiently large, $p(n)=\ell_R(I^n/J^n)$. The Amao multiplicity $a(J,I)$ is defined to be $d!$ times the coefficient of the degree $d$ term of $p(n)$. That is,
$$
a(J,I)=\lim_{n\rightarrow\infty}\frac{\ell_R(I^n/J^n)}{n^d/d!}.
$$
The Amao multiplicity is always a nonnegative integer.

In Theorem \ref{TheoremA}, we prove that the epsilon multiplicity of an ideal in an analytically unramified $d$-dimensional local ring is a limit of Amao multiplicities.

\begin{theorem}\label{TheoremA} Let $R$ be an analytically unramified $d$-dimensional local ring and $I\subset R$ be an ideal. Then
$$
\epsilon(I)=
\lim_{m\rightarrow\infty}\frac{a(I^m,(I^m)^{\rm sat})}{m^d}
$$
\end{theorem}

The proof of Theorem \ref{TheoremA} is in Section \ref{SecEM}.
Theorem \ref{TheoremA} is a consequence of Theorem \ref{MultFormula},
which is proven in Section \ref{SecMT}.

We obtain the following corollary, since with the  assumptions of the corollary, $P^{(n)}=(P^n)^{\rm sat}$ for all $n$.

\begin{Corollary}\label{corollaryB} Let $R$ be a $d$-dimensional regular local ring and $P\subset R$ be a prime ideal such that $R/P$ has an isolated singularity. Then 
$$
\lim_{n\rightarrow\infty}\frac{\ell_R(P^{(n)}/P^n)}{n^d/d!}=\epsilon(P)=
\lim_{m\rightarrow\infty}\frac{a(P^{m},P^{(m)})}{m^d}
$$
\end{Corollary}

A graded family of ideals $\mathcal I=\{I_n\}$ in a local ring $R$ is a family of ideals indexed by the natural numbers such that $I_0=R$ and $I_mI_n\subset I_{m+n}$ for all $m,n\in \N$. The family $\mathcal I$ is called a graded family of $m_R$-primary ideals if in addition, $I_n$ is $m_R$-primary for $n>0$.

Multiplicities of graded families of $m_R$-primary ideals are computed in increasing generality in \cite{ELS}, \cite{M1}, \cite{LM}, \cite{C1} and \cite{C2}. The last three of these papers use the method of volumes of Okounkov bodies to make this computation. The method of volumes of Okounkov bodies was initially introduced to compute volumes of graded linear systems in \cite{Ok}, \cite{KK} and \cite{LM}. 

Theorem \ref{TheoremA} follows from the more general Theorem \ref{MultFormula}. The existence of the limit
$$
F(\mathcal I,\mathcal J)=\lim_{n\rightarrow\infty}\frac{\ell_R(J_n/I_n)}{n^d/d!}
$$
in Theorem \ref{MultFormula} follows from Theorem 6.1 \cite{C2}.
The proof of Theorem \ref{MultFormula}  makes a reduction to the case where $R$ is analytically irreducible. This reduction argument is a generalization of the one in the proof of Theorem 6.1 \cite{C2}.

 To compute the epsilon multiplicity of an ideal $I$ in an analytically irreducible ring,  we must compute associated semigroups to $I$ and $I^{\rm sat}$, truncated at some value $\beta$ sufficiently large that $I$ and $I^{\rm sat}$ agree at their elements of larger value, and also so that $\epsilon(I)$ can be computed by taking a difference of the associated volumes of the Okounkov bodies associated to their semigroups, truncated at $\beta$.

Theorem \ref{MultFormula} involves an infinite number of graded families of ideals, $\mathcal I$ and $\mathcal J$, and for all $m\ge 1$, $\mathcal I(m)$ and $\mathcal J(m)$.
A subtlety of the proof of Theorem \ref{MultFormula} (in the case that $R$ is analytically irreducible) is that we must find a $\beta$ such that the truncation at $\beta$ works for
 the graded systems of ideals $\mathcal I$ and $\mathcal J$,  and $m\beta$ works for 
 the filtrations $\mathcal I(m)$ and $\mathcal J(m)$ for all $m\in \N$. The theoretical realization of this in the theory  of Okounkov bodies is achieved in Theorem \ref{TheoremB1} and Theorem \ref{TheoremB2} of Section \ref{SecGF}.
 
 We thank Antoni Rangachev for raising the interesting question of  if Theorem \ref{TheoremA} is true.

\section{Semigroups and Cones}

Suppose that $S\subset \N^{d+1}$ is a semigroup. Let $\Sigma=\Sigma(S)$ be the closed convex cone in $\R^{d+1}$ which is the closure of the set of all linear combinations $\sum\lambda_is_i$ with $\lambda_i\in \R_{\ge 0}$ and $s_i\in S$.
Set 
$$
\Delta=\Delta(S)=\Sigma\cap(\R^d\times\{1\}).
$$
For $m\in \N$, put
$$
S_m=S\cap (\N^d\times\{m\}).
$$
which can be viewed as a subset of $\N^d$. Let $G(S)$ be the subgroup of $\Z^{d+1}$ generated by $S$.
\begin{theorem}\label{TheoremL}(Okounkov,  Section 3 \cite{Ok}, Lazarsfeld and Musta\c{t}\u{a}, Proposition 2.1 \cite{LM}, Kaveh and Khovanskii \cite{KK}, Theorem 3.2 \cite{C2}) Suppose that a subsemigroup $S$ of  $\Z^d\times\N$  satisfies the following two conditions:
\begin{equation}\label{Cone2}
\begin{array}{l}
\mbox{There exist finitely many vectors $(v_i,1)$ spanning a semigroup $B\subset\N^{d+1}$}\\
\mbox{such that $S\subset B$}
\end{array}
\end{equation}
and
\begin{equation}\label{Cone3}
G(S)=\Z^{d+1}.
\end{equation}
Then
$$
\lim_{n\rightarrow\infty} \frac{\# S_n}{n^d}={\rm vol}(\Delta(S)).
$$
\end{theorem}

\begin{theorem}\label{ConeTheorem2}(Proposition 3.1 \cite{LM}) Suppose that $S$ satisfies (\ref{Cone2}) and (\ref{Cone3}). Fix $\epsilon>0$. Then there is an integer $p_0= p_0(\epsilon)$
such that if $p\ge p_0$, then the limit
$$
\lim_{k\rightarrow \infty} \frac{\#(k S_p)}{k^dp^d}\ge {\rm vol}(\Delta(S))-\epsilon
$$
exists, where
$$
k S_p=\{x_1+\cdots+x_k\mid x_1,\ldots,x_k\in S_p\}.
$$
\end{theorem}

\section{Graded families of ideals and volumes}\label{SecGF}

In this section, we develop  methods of computation of multiplicities as volumes, which we will need for the proof of Theorem \ref{MultFormula} in Section \ref{SecMT}.

\begin{lemma}\label{DimlessthandLemma} Suppose that $d$ is a positive integer and $R$ is a Noetherian local ring such that $\dim R < d$. Let $\mathcal{I} = \{I_i\}$ and $\mathcal{J} = \{J_i\}$ be graded families of ideals on $R$ such that $I_i\subset J_i$ for all $i$, and there exists a positive integer $c$ such that
\begin{equation*}
    I_i\cap m_R^{ci} = J_i\cap m_R^{ci}
\end{equation*}for all $i$. Then $\lim_{n\to\infty}\ell_R(J_n/I_n)/n^d=0$.
\end{lemma}
\begin{proof}
    Since $I_i\cap m_R^{ci} = J_i\cap m_R^{ci}$, we have a short exact sequence
    \begin{equation*}
        0\to \frac{I_i}{I_i\cap m_R^{ci}}\to\frac{J_i}{J_i\cap m_R^{ci}}\to \frac{J_i}{I_i}\to 0
    \end{equation*}for all $i$, so that
    \begin{equation}\label{lem_eq1}
        \ell_R(J_i/I_i) = \ell_R(J_i/J_i\cap m_R^{ci}) - \ell_R(I_i/I_i\cap m_R^{ci}).
    \end{equation}
    Since $\dim R\leq d-1$, there exist positive integers $N$ and $\gamma$ such that $\ell_R(R/m_R^{ci})\leq \gamma i^{d-1}$ for $i\geq N$. We also have an injective map $J_i/J_i\cap m_R^{ci}\to R/m_R^{ci}$. Hence,
    \begin{equation*}
        \lim_{n\to\infty}\frac{\ell_R(J_n/J_n\cap m_R^{cn})}{n^d}\leq \lim_{n\to\infty}\frac{\ell_R(R/m_R^{cn})}{n^d}
        \leq \lim_{n\to\infty}\frac{\gamma n^{d-1}}{n^d}=0.
    \end{equation*}
    Similarly, $\lim_{n\to\infty}\ell_R(I_n/I_n\cap m_R^{cn})/n^d=0$, and hence, $\lim_{n\to\infty}\ell_R(J_n/I_n)=0$ by (\ref{lem_eq1}).
\end{proof}

We first summarize a construction from Section 4 of \cite{C2}.
For the remainder of this section, 
let $R$ be a complete (Noetherian) local domain and
$\pi:X\rightarrow \mbox{spec}(R)$ be the normalization of the blow up of the maximal ideal $m_R$ of $R$. The scheme $X$ is of finite type over $R$ since $R$ is excellent.
Since $\pi^{-1}(m_R)$ has codimension 1 in $X$ and $X$ is normal, there exists a closed point $x\in X$ such that the local ring $\mathcal O_{X,x}$ is a regular local ring. Let $S$ be this local ring. $S$ is a regular local ring which  is essentially of finite type  and birational over $R$ ($R$ and $S$ have the same  quotient field ${\rm Q}(R)$).

Let $y_1,\ldots,y_d$ be a regular system of parameters in $S$. Let $\lambda_1,\ldots,\lambda_d$ be rationally independent real numbers, such that 
\begin{equation}\label{eq9}
\lambda_i\ge 1\mbox{ for all $i$}.
\end{equation}
 We define a valuation $\nu$ on
$Q(R)$ which dominates $S$ by prescribing 
$$
\nu(y_1^{a_1}\cdots y_d^{a_d})=a_1\lambda_1+\cdots+a_d\lambda_d
$$
for $a_1,\ldots,a_d\in \Z_+$, and $\nu(\gamma)=0$  if $\gamma\in S$ has nonzero residue in $S/m_S$.

Let $C$ be a coefficient set of $S$. Since $S$ is a regular local ring, for $r\in \Z_+$ and $f\in S$, there is a unique expression 
\begin{equation}\label{eqred11}
f=\sum s_{i_1,\ldots,i_d}y_1^{i_1}\cdots y_d^{i_d}+g_r
\end{equation}
with $g_r\in m_S^r$, $s_{i_1,\ldots,i_d}\in S$ and $i_1+\cdots+i_d<r$ for all $i_1,\ldots,i_d$ appearing in the sum. Take $r$ so large that 
$r> i_1\lambda_1+\cdots+i_d\lambda_d$ for some term with $s_{i_1,\ldots,i_d}\ne 0$. Then define
\begin{equation}\label{eqG61}
\nu(f)=\min\{i_1\lambda_1+\cdots+i_d\lambda_d\mid s_{i_1,\ldots,i_d}\ne 0\}.
\end{equation}
This definition is well defined, and we calculate that
$\nu(f+g)\ge \min\{\nu(f),\nu(g)\}$ and $\nu(fg)=\nu(f)+\nu(g)$ (by the uniqueness of the minimal value term in the expansion (\ref{eqred11})) for all $0\ne f,g\in S$. Thus $\nu$ is a valuation.
 Let $V_{\nu}$ be the valuation ring of $\nu$ in $Q(R)$. The value group $\Gamma_{\nu}$ of $V_{\nu}$ is the (nondiscrete) ordered subgroup 
$\Z\lambda_1+\cdots+\Z\lambda_d$ of $\R$. Since there is a unique monomial giving the minimum in (\ref{eqG61}), we have that the residue field of $V_{\nu}$ is
$S/m_S$.

Let $k=R/m_R$ and $k'=S/m_S=V_{\nu}/m_{\nu}$. Since $S$ is essentially of finite type over $R$, we have that $[k':k]<\infty$.

For $\lambda\in \R$, define  ideals $K_{\lambda}$ and $K_{\lambda}^+$ in $V_{\nu}$ by
$$
K_{\lambda}=\{f\in Q(R)\mid \nu(f)\ge\lambda\}
$$
and
$$
K_{\lambda}^+=\{f\in Q(R)\mid \nu(f)>\lambda\}.
$$

We follow the usual convention that $\nu(0)=\infty$ is larger than any element of $\R$. By Lemma 4.3 \cite{C1}, we have the following formula. The fact that $R$ is analytically irreducible is necessary for the validity of the formula.

\begin{equation}\label{eqred50}
\mbox{There exists $\alpha\in \Z_+$ such that $K_{\alpha n}\cap R\subset m_R^n$ for all $n\in \N$.}
\end{equation}


Let $\mathcal A=\{A_n\}$ be a graded family of ideals on $R$.

For $\beta\in \Z_{>0}$ and $t\ge  1$, define 
$$
\Gamma_{\beta}(\mathcal A)^{(t)}=\left\{
\begin{array}{l}
(n_1,\ldots,n_d,i)\in \N^{d+1}\mid \dim_k A_i\cap K_{n_1\lambda_1+\cdots+n_d\lambda_d}/A_i\cap K_{n_1\lambda_1+\cdots+n_d\lambda_d}^+\ge t\\
\mbox{ and }n_1\lambda_1+\cdots+n_d\lambda_d\le \beta i
\end{array}
\right\}.
$$

Define $\Gamma_{\beta}(\mathcal A)=\Gamma_{\beta}(\mathcal A)^{(1)}$. 

Let $\lambda=n_1\lambda_1+\cdots+n_d\lambda_d$ be such that $n_1\lambda_1+\cdots+n_d\lambda_d\le\beta i$. Then
\begin{equation}\label{eqred40}
\dim_kK_{\lambda}\cap A_i/K_{\lambda}^+\cap A_i=\#\{t|(n_1,\ldots,n_d,i)\in \Gamma_{\beta}(\mathcal A)^{(t)}\}.
\end{equation}

Let $s=[k':k]$. We have (by formula (15) \cite{C2} or the formula before Lemma 5.2 \cite{C3})  that 
\begin{equation}\label{eq1}
\Gamma_{\beta}(\mathcal A)^{(t)}=\emptyset\mbox{ for }t>s.
\end{equation}

Write $y_i=\frac{f_i}{g}$ with $f_i,g\in R$ for $1\le i\le d$.

For $0\ne f\in R$, define $\phi(f)=(n_1,\ldots,n_d)\in \N^d$ and $\psi(f)=v(f)$ if $v(f)=n_1\lambda_1+\cdots+n_d\lambda_d$.

\begin{lemma}\label{Lemma1} Let $\mathcal A=\{A_n\}$ be a graded family of ideals of $R$ and suppose that $A_1\ne 0$. Let $0\ne x\in A_1$ and suppose that 
$$
\beta\ge \beta_0(\mathcal A):=\max\{(\psi(xf_1),\ldots,\psi(xf_d), \psi(xg)\}.
$$
Then $\Gamma_{\beta}(\mathcal A)$ satisfies (\ref{Cone2}) and (\ref{Cone3}) of Theorem \ref{TheoremL}. 
\end{lemma}

\begin{proof}    Let $\{e_1,\ldots,e_d\}$ be the standard basis of $\Z^d$. Let $\lambda=\min\{\lambda_1,\ldots,\lambda_d\}$.
The semigroup 
$$
B:=\{(n_1,\ldots,n_d,i)\mid (n_1,\ldots,n_d,i)\in \N^{d+1}\mbox{ and }
n_1+\cdots+n_d\le i\lceil \frac{\beta}{\lambda}\rceil\}
$$
is generated by $B\cap (\N^d\times\{1\})$ and  contains $\Gamma_{\beta}(\mathcal A)$ so (\ref{Cone2}) of Theorem \ref{TheoremL} holds. 

Let $G=G(\Gamma_{\beta}(\mathcal A))$.
We have that $xf_1,\ldots,xf_d,xg\in A_1$ so 
$$
(\phi(xf_1),1),\ldots,(\phi(xf_d),1), (\phi(xg),1)\in G.
$$
Thus $(e_i,0)=(\phi(f_i)-\phi(g),0)\in G$ for $1\le i\le d$, and since $(\phi(xg),1)\in G$, we then have that $(0,\ldots,0,1)\in G$ and so (\ref{Cone3}) of Theorem \ref{TheoremL} holds.
\end{proof}

\begin{corollary}\label{Cor3} Suppose that $0\ne x\in A_1$ and $\beta_0(\mathcal A)$ of Lemma \ref{Lemma1} is calculated from $x$. Suppose that $\mathcal B=\{B_i\}$ is a graded family of ideals on $R$ such that $\mathcal A\subset \mathcal B$ ($A_i\subset B_i$ for all $i$)
 and $\beta\ge \beta_0(\mathcal A)$. Then $\Gamma_{\beta}(\mathcal B)$ satisfies (\ref{Cone2}) and (\ref{Cone3}) of Theorem \ref{TheoremL}.
\end{corollary}

\begin{proof} Since $x\in B_1$, the calculation of $\beta_0(\mathcal B)$ using $x$ in Lemma \ref{Lemma1} is $\beta_0(\mathcal B)=\beta_0(\mathcal A)$.
\end{proof}

\begin{corollary} \label{Cor1} Suppose that $0\ne x\in A_1$ and $\beta_0(\mathcal A)$ of Lemma \ref{Lemma1} is calculated from $x$.
Suppose that  $\beta\ge m\beta_0(\mathcal A)$. 
Let $\mathcal C=\{C_i\}$ be a graded family of ideals such that $A_1^m\subset C_1$. Then 
$\Gamma_{\beta}(\mathcal C)$ satisfies (\ref{Cone2}) and (\ref{Cone3}) of Theorem \ref{TheoremL}. 
\end{corollary}

\begin{proof} We have that  $x^m\in A_1^m\subset C_1$. By Lemma \ref{Lemma1}, $\Gamma_{\beta}(C)$ satisfies (\ref{Cone2}) and (\ref{Cone3}) of Theorem \ref{TheoremL} for 
$$
\beta\ge \max\{\psi(x^mf_1),\ldots,\psi(x^mf_d),\psi(x^mg)\}.
$$
Now 
$$
\begin{array}{lll}
m\beta_0(\mathcal A)&=&m\max\{\psi(xf_1),\ldots,\psi(xf_d),\psi(xg)\}\\
&=&\max\{\psi(x^mf_1^m),\ldots,\psi(x^mf_d^m),\psi(x^mg^m)\}\\
&\ge& \max\{\psi(x^mf_1),\ldots,\psi(x^mf_d),\psi(x^mg)\}
\end{array}
$$
so the conclusions of the corollary holds.

\end{proof}

\begin{lemma}\label{Lemma4}(Lemma 4.4 \cite{C2}, Lemma 5.3 \cite{C3}) Suppose that  $0\ne f\in A_i$, $0\ne g\in A_j$ and
$$
\dim_kA_i\cap K_{\nu(f)}/A_i\cap K_{\nu(f)}^+=s.
$$
Then
\begin{equation}\label{eqred10}
\dim_k A_{i+j}\cap K_{\nu(fg)}/A_{i+j}\cap K_{\nu(fg)}^+=s.
\end{equation}
\end{lemma}

\begin{lemma}\label{lemma3}(Lemma 5.4 \cite{C3}) Suppose that $\Gamma_{\beta}(\mathcal A)^{(s)}\ne \emptyset$. Then $\Delta(\Gamma_{\beta}(\mathcal A)^{(s)})=\Delta(\Gamma_{\beta}(\mathcal A))$.
\end{lemma}

\begin{lemma}\label{lemma2} (Lemma 5.5 \cite{C3}) There exists $a\in R$ such that $\dim_kR\cap K_{v(a)}/R\cap K_{v(a)}^+=s$.
\end{lemma}

\begin{proposition}\label{Prop4} Let $0\ne x\in A_1$ and $\beta_1(\mathcal A):=\psi(a)+\psi(x)$. Then $\Gamma_{\beta}(\mathcal A)^{(s)}\ne\emptyset$ for $\beta\ge \beta_1(\mathcal A)$.
\end{proposition}

\begin{proof} We have that $ax\in A_1$ and $\psi(ax)=\beta_1(\mathcal A)$.
By Lemma \ref{Lemma4},
\begin{equation}\label{eq2}
\dim_k (A_1\cap K_{v(ax)}/A_1\cap K_{v(ax)}^+)=s.
\end{equation}
 Hence $\Gamma_{\beta}(\mathcal A)^{(s)}\ne \emptyset$ for $\beta\ge \beta_1(\mathcal A)$.
\end{proof}

\begin{corollary}\label{Cor5} Let $0\ne x\in A_1$ and $\beta_1(\mathcal A):=\psi(a)+\psi(x)$. Suppose that $\mathcal B=\{B_i\}$ is a graded family of ideals such that $\mathcal A\subset \mathcal B$ and $\beta\ge \beta_1(\mathcal A)$.
Then $\Gamma_{\beta}(\mathcal B)^{(s)}\ne\emptyset$ for $\beta\ge \beta_1(\mathcal A)$.
\end{corollary}

\begin{proof} This follows from Proposition \ref{Prop4} since $x\in B_1$.
\end{proof}

\begin{proposition}\label{Prop1} Let $0\ne x\in A_1$ and $\beta_1(\mathcal A):=\psi(a)+\psi(x)$.
Let $\mathcal C=\{C_i\}$ be a graded family of ideals such that $A_1^m\subset C_1$.
Then for $m\ge 1$, $\Gamma_{\beta}(\mathcal C)^{(s)}\ne \emptyset$ for $\beta\ge m\beta_1(\mathcal A)$.
\end{proposition}

\begin{proof}
Let $h_m=ax^m\in A_1^m\subset C_1$. By Lemma \ref{Lemma4},
\begin{equation}
\dim_k (A_m\cap K_{v(h_m)}/A_m\cap K_{v(h_m)}^+)=s.
\end{equation}
Then 
$\psi(ax^m)=\psi(a)+m\psi(x)\le m\beta_1(\mathcal A)$. Hence $\Gamma_{\beta}(\mathcal C)^{(s)}\ne \emptyset$ for $\beta\ge m\beta_1(\mathcal A)$.
\end{proof}



We obtain the following strengthening of Theorem 5.6 \cite{C3}.
\begin{theorem}\label{TheoremB0} Suppose that $0\ne x\in A_1$ and $\beta_0(\mathcal A)$ of Lemma \ref{Lemma1} 
and $\beta_1(\mathcal A)$ of Proposition \ref{Prop4} are calculated from $x$. 
Let $\beta_2(\mathcal A)=\max\{\beta_0(\mathcal A),\beta_1(\mathcal A)\}$.
Then for $\beta\ge \beta_2(\mathcal A)$,
$$
\lim_{n\rightarrow\infty}\frac{\ell_R(A_n/A_n\cap K_{\beta n})}{n^d}
=[k':k]{\rm vol}(\Delta(\Gamma_{\beta}(\mathcal A))).
$$
\end{theorem}

\begin{proof} We start with a variation on formula (32) \cite{C3}. We have 
by (\ref{eqred40}) and (\ref{eq1})  that
\begin{equation}\label{eq3}
\ell_R(A_i/A_i\cap K_{\beta i})=\sum_{0\le \lambda<\beta i}
\dim_k(K_{\lambda}\cap A_i/K_{\lambda}^+\cap A_i)=\sum_{t=1}^{[k':k]}\#(\Gamma_{\beta}(\mathcal A)^{(t)}).
\end{equation}
We have that $\#(\Gamma_{\beta}(\mathcal A)^{(t)})\ne\emptyset$ for $1\le t\le s$ by Proposition \ref{Prop4}, so
$$
\Delta(\Gamma_{\beta}(\mathcal A)^{(t)})=\Delta(\Gamma_{\beta}(\mathcal A))
$$
for $1\le t\le [k':k]$ by Lemma \ref{lemma3}.
Now $\Gamma_{\beta}(\mathcal A)$ satisfies the conditions (\ref{Cone2}) and (\ref{Cone3}) of Theorem \ref{TheoremL} by Lemma \ref{Lemma1}, so that 
$$
\lim_{n\rightarrow\infty}\frac{\#(\Gamma_{\beta}(\mathcal A))}{n^d}={\rm vol}(\Delta(\Gamma_{\beta}(\mathcal A))
$$
by Theorem \ref{TheoremL}, from which we obtain the conclusions of Theorem \ref{TheoremB0}.
\end{proof}

\begin{theorem}\label{TheoremB1} Suppose that $0\ne x\in A_1$ and $\beta_0(\mathcal A)$ of Lemma \ref{Lemma1} 
and $\beta_1(\mathcal A)$ of Proposition \ref{Prop4} are calculated from $x$. 
Let $\beta_2(\mathcal A)=\max\{\beta_0(\mathcal A),\beta_1(\mathcal A)\}$.
Suppose that $\mathcal B=\{B_i\}$ is a graded family of ideals such that $\mathcal A\subset \mathcal B$.
Then for $\beta\ge \beta_2(\mathcal A)$,
$$
\lim_{n\rightarrow\infty}\frac{\ell_R(B_n/B_n\cap K_{\beta n})}{n^d}
=[k':k]{\rm vol}(\Delta(\Gamma_{\beta}(\mathcal B))).
$$
\end{theorem}

\begin{proof} The proof follows from the proof of Theorem \ref{TheoremB0}, replacing $\mathcal A=\{A_i\}$ with $\mathcal B=\{B_i\}$, the reference Proposition \ref{Prop4} with Corollary \ref{Cor5} and the reference Lemma \ref{Lemma1} with Corollary \ref{Cor3}.
\end{proof}

\begin{theorem}\label{TheoremB2} Suppose that $0\ne x\in A_1$ and $\beta_0(\mathcal A)$ of Lemma \ref{Lemma1} 
and $\beta_1(\mathcal A)$ of Proposition \ref{Prop4} are calculated from $x$. 
Let $\beta_2(\mathcal A)=\max\{\beta_0(\mathcal A),\beta_1(\mathcal A)\}$.
Let $\mathcal C=\{C_i\}$ be a graded family of ideals such that $A_1^m\subset C_1$.
Then for $\beta\ge m\beta_2(\mathcal A)$,
$$
\lim_{n\rightarrow\infty}\frac{\ell_R(C_n/C_n\cap K_{\beta n})}{n^d}
=[k':k]{\rm vol}(\Delta(\Gamma_{\beta}(\mathcal C))).
$$
\end{theorem}

\begin{proof} The proof follows from the proof of Theorem \ref{TheoremB0}, replacing $\mathcal A=\{A_i\}$ with $\mathcal C=\{C_i\}$, the reference Proposition \ref{Prop4} with Proposition \ref{Prop1} and the reference Lemma \ref{Lemma1} with Corollary \ref{Cor1}.
\end{proof}

\section{The main technical theorem}\label{SecMT}

In this section, we prove the following theorem, from which Theorem \ref{TheoremA} will follow.

\begin{Theorem}\label{MultFormula}
    Suppose that $R$ is an analytically unramified (Noetherian) local ring of dimension $d>0$ and that $\mathcal{I} = \{I_n\}$ and $\mathcal{J} = \{J_n\}$ are graded families of ideals in $R$ such that $I_n\subset J_n$ for $n\geq 1$. Suppose that for $m\in\mathbb{Z}_{>0}$, there exist graded families of ideals $\mathcal{I}(m) = \{I(m)_i\}$ and $\mathcal{J}(m) = \{J(m)_i\}$ such that $I(m)_i\subset J(m)_i$, $I(m)_1=I_m$ and $I(m)_i\subset I_{mi}$ for all $i$, $J(m)_1=J_m$, and $J(m)_i\subset J_{mi}$ for all $i$. 
    Suppose that $I_i\subset P$ for all $i$ if $P$ is a minimal prime of $R$ such that $\dim(R/P)=d$ and $I_1\subset P$. Suppose further that there exists $c\in\mathbb{Z}_{>0}$ such that
\begin{equation*}
    I(m)_i\cap m_R^{cim} = J(m)_i\cap m_R^{cim}
\end{equation*} for all $i$ and $m$. Then the limit
\begin{equation*}
    F(\mathcal{I},\mathcal{J}) = \lim_{n\to\infty}\frac{\ell_R(J_n/I_n)}{n^d/d!}
\end{equation*}exists, and for $m\in\mathbb{Z}_{>0}$, the limits
\begin{equation*}
    G(\mathcal{I}(m),\mathcal{J}(m)) = \lim_{k\to\infty}\frac{\ell_R(J(m)_k/I(m)_k)}{k^d/d!}
\end{equation*}exist. Further, the limit
\begin{equation*}
    \lim_{m\to\infty}\frac{G(\mathcal{I}(m),\mathcal{J}(m))}{m^d}
\end{equation*}exists and,
\begin{equation*}
    F(\mathcal{I},\mathcal{J}) = \lim_{m\to\infty}\frac{G(\mathcal{I}(m),\mathcal{J}(m))}{m^d}.
\end{equation*}
\end{Theorem}

We see that the condition on minimal primes is automatic in the case that $\{I_n\}$ is a filtration.

For the remainder of this section, let assumptions be as in the statement of Theorem \ref{MultFormula}.

\begin{Remark}\label{RemarkMinimal}
    Suppose that $P$ is a minimal prime of $R$ such that $\dim(R/P)=d$ and $I_1\subset P$. Then $I_i,J_i\subset P$ for each $i$, and $I(m)_k,J(m)_k\subset P$ for all $m$ and $k$.
\end{Remark}
\begin{proof}
    $I_i\subset P$ by assumption. We also have that $I(m)_k\subset I_{mk}\subset P$. Moreover, $J(m)_k\neq R$, otherwise $m_R^{cmk} =I(m)_k\cap m_R^{cmk}\subset P$, so that $m_R\subset P$, contradicting the fact that $\dim(R)\neq 0$. Thus, $J(m)_k^{cmk}\subset J(m)_k\cap m_R^{cmk}= I(m)_k\cap m_R^{cmk}\subset P$ so that $J(m)_k\subset P$. In particular, $J_i = J(i)_1\subset P$.
\end{proof}

To prove Theorem \ref{MultFormula}, we first reduce to the case that $R$ is complete.
By flatness of $R\to \widehat{R}$, we have
\begin{equation*}
\begin{split}
    I(m)_k\widehat{R}\cap m_{\widehat{R}}^{ckm}
    &= I(m)_k\widehat{R}\cap m_R^{ckm}\widehat{R}
    =(I(m)_k\cap m_R^{ckm})\widehat{R}
    \\&=(J(m)_k\cap m_R^{ckm})\widehat{R}
    = J(m)_k\widehat{R} \cap m_{\widehat{R}}^{ckm}.
\end{split}
\end{equation*}
We will show that if $Q$ is a minimal prime of $\widehat{R}$ such that $\dim(\widehat{R}/Q)=d$, and $\widehat{I_1}\subset Q$, then $\widehat{I_i}\subset Q$ for all $i$. We have $I_1\subset \widehat{I_1}\cap R\subset Q\cap R$, while $Q\cap R$ is a prime ideal of $R$. Since $(Q\cap R)\widehat{R}\subset Q$, we have that
\begin{equation*}
\begin{split}
    d&=\dim R\geq\dim R/Q\cap R
    =\dim \widehat{R/Q\cap R}
    =\dim \widehat{R}/(Q\cap R)\widehat{R}
    \\&\geq \dim \widehat{R}/Q= d
\end{split}
\end{equation*}which shows that $Q\cap R$ is a minimal prime of $R$ such that $\dim R/Q\cap R = d$. Now, since $I_1\subset Q\cap R$, $I_i\subset Q\cap R$ so that $\widehat{I_i}=I_i\widehat{R}\subset (Q\cap R)\widehat{R}\subset Q$ for all $i$. Thus the assumptions of Theorem \ref{MultFormula}  hold for the graded family $\widehat{\mathcal{I}} = \Big\{\widehat{I}_i\Big\}$, and by Remark \ref{RemarkMinimal}, they also hold for 
$\widehat{\mathcal{J}}=\Big\{\widehat{J}_i\Big\}$, $\widehat{\mathcal{I}}(m) = \Big\{\widehat{I(m)_k}\Big\}$, and $\widehat{\mathcal{J}}(m) = \Big\{\widehat{J(m)_k}\Big\}$ in $\widehat{R}$. Since $m_R^{ckm}(J(m)_k/I(m)_k)=0$ for all $m$ and $k$, we have that
\begin{equation*}
    \ell_{\widehat{R}}\bigg(\widehat{J(m)_k}/\widehat{I(m)_k}\bigg)
    = \ell_{R}(J(m)_k/I(m)_k).
\end{equation*}
for $m,k\geq 1$. Thus we may assume that $R$ is complete in the statement of Theorem \ref{MultFormula} by replacing $R$, $I_n$, $J_n$, $I(m)_k$, and $J(m)_k$ with $\widehat{R}$, $I_n\widehat{R}$, $J_n\widehat{R}$, $I(m)_k\widehat{R}$, and $J(m)_k\widehat{R}$ respectively.\\

Now we prove Theorem \ref{MultFormula} in the case that $R$ is a complete domain. There are two cases which we will consider separately, $I_1=0$ and $I_1\ne 0$. If $I_1=0$, then $I_i=0$ for all $i$ by assumption and so
\begin{equation*}
    I_i,J_i,I(m)_k,J(m)_k=0
\end{equation*}for all $i$, $m$, and $k$ by Remark \ref{RemarkMinimal}. Now $F(\mathcal{I},\mathcal{J}) = d!\lim_{n\to\infty}\ell_R(0/0)/n^d \\= d!\lim_{n\to\infty}0=0$, and $G(\mathcal{I}(m),\mathcal{J}(m)) = d!\lim_{k\to\infty}\ell_R(0/0)/k^d=0$ for all $m$. Thus,
\begin{equation*}
    \lim_{m\to\infty}\frac{G(\mathcal{I}(m),\mathcal{J}(m))}{m^d}=\lim_{m\to\infty}0/m^d=0
\end{equation*}and equals $F(\mathcal{I},\mathcal{J})$.\\

Now assume that  $I_1\ne 0$, so that $J_1\ne 0$, $I(m)_1\ne 0$ for all $m$ and $J(m)_1\ne 0$ for all $m$.

By Theorems \ref{TheoremB0}, \ref{TheoremB1} and \ref{TheoremB2}, $\beta\in \Z_{>0}$ can be chosen large enough that 
\begin{equation}\label{Leq1}
\lim_{n\rightarrow\infty} \frac{\ell_R(I_n/I_n\cap K_{\beta n})}{n^d}=[k':k]{\rm vol}(\Delta(\Gamma_{\beta}(\mathcal I)),
\end{equation}

\begin{equation}\label{Leq2}
\lim_{n\rightarrow\infty} \frac{\ell_R(J_n/J_n\cap K_{\beta n})}{n^d}=[k':k]{\rm vol}(\Delta(\Gamma_{\beta}(\mathcal J)),
\end{equation}

\begin{equation}\label{Leq3}
\lim_{k\rightarrow\infty} \frac{\ell_R(I(m)_k/I(m)_k\cap K_{\beta mk})}{k^d}=[k':k]{\rm vol}(\Delta(\Gamma_{m\beta}(\mathcal I(m)))
\end{equation}
and
\begin{equation}\label{Leq4}
\lim_{k\rightarrow\infty} \frac{\ell_R(J(m)_k/J(m)_k\cap K_{\beta mk})}{k^d}=[k':k]{\rm vol}(\Delta(\Gamma_{m\beta}(\mathcal J(m))).
\end{equation}
We further require that  $\beta\ge \alpha c$, where $\alpha$ is the constant of  (\ref{eqred50}), so that 
\begin{equation}\label{Leq5}
K_{\beta n}\cap R\subset m_R^{cn}
\end{equation}
for all $n$. Thus
$$
\ell_R(J_n/I_n)=\ell_R(J_n/J_n\cap K_{\beta n})-\ell_R(I_n/I_n\cap K_{\beta n})
$$
for all $n$, and so by (\ref{Leq2}) and (\ref{Leq1}), 
\begin{equation}\label{Leq6}
\lim_{n\rightarrow\infty}\frac{\ell_R(J_n/I_n)}{n^d}=[k':k]\left({\rm vol}(\Delta(\Gamma_{\beta}(\mathcal J)))-{\rm vol}(\Delta(\Gamma_{\beta}(\mathcal I)))\right).
\end{equation}
Thus the limit 
$$
F(\mathcal I,\mathcal J)=d!\lim_{n\rightarrow\infty}\frac{\ell_R(J_n/I_n)}{n^d}
$$
exists.

By (\ref{Leq5}), we have that 
$$
\ell_R(J(m)_k/I(m)_k)=\ell_R(J(m)_k/J(m)_k\cap K_{\beta m k})-\ell_R(I(m)_k/I(m)_k\cap K_{\beta m k})
$$
for all $m$ and $k$. Thus by (\ref{Leq3}) and (\ref{Leq4}), for all $m$,
\begin{equation}\label{Leq7}
\lim_{k\rightarrow\infty}\frac{\ell_R(J(m)_k/I(m)_k)}{k^d}
=[k':k]\left({\rm vol}(\Delta(\Gamma_{\beta m}(\mathcal J(m))))-{\rm vol}(\Delta(\Gamma_{\beta m}(\mathcal I(m))))\right)
\end{equation}
and so the limits
$$
G(\mathcal I(m),\mathcal J(m))=d!\lim_{k\rightarrow\infty}\frac{\ell_R(J(m)_k/I(m)_k)}{k^d}
$$
exist.

The definition of $k*S_m$ for a subsemigroup $S$ of $\N^{d+1}$ is given in Theorem \ref{ConeTheorem2}. Let $\pi:\N^{d+1}\rightarrow \N^d$ be projection onto the first $d$-factors. 
Since $I_m^k\subset I(m)_k\subset I_{mk}$ for all $m$ and $k$, 
$$
\pi(k*\Gamma_{\beta}(\mathcal I)_m)\subset \pi(\Gamma_{\beta m}(\mathcal I(m)_k))\subset \pi(\Gamma_{\beta}(\mathcal I)_{mk}).
$$
so that
$$
\#(k*\Gamma_{\beta}(\mathcal I)_m)\le \#(\Gamma_{\beta m}(\mathcal I(m)_k))\le\#(\Gamma_{\beta}(\mathcal I)_{mk}).
$$

By our choice of $\beta$, $\Gamma_{\beta}(\mathcal I)$ satisfies (\ref{Cone2}) and (\ref{Cone3}), so by 
Theorem \ref{ConeTheorem2}, for given $\epsilon>0$, there exists $m_0>0$ such that 
$$
{\rm vol}(\Delta(\Gamma_{\beta}(\mathcal I)))-\epsilon\le
\lim_{k\rightarrow \infty}\frac{\#(k*\Gamma_{\beta}(\mathcal I)_m)}{k^dm^d}.
$$
Thus
$$
\begin{array}{lll}
{\rm vol}(\Delta(\Gamma_{\beta}(\mathcal I)))-\epsilon&\le&
\lim_{k\rightarrow \infty}\frac{\#(\Gamma_{\beta m}(\mathcal I(m)_k))}{k^dm^d}
=\frac{{\rm vol}(\Delta(\Gamma_{\beta m}(\mathcal I(m))))}{m^d}\\
&\le&\lim_{k\rightarrow\infty}\frac{\#(\Gamma_{\beta}(\mathcal I)_{mk})}{k^dm^d}
={\rm vol}(\Delta(\Gamma_{\beta}(\mathcal I)))
\end{array}
$$
for $m\ge m_0$.

Similarly, after possibly increasing $m_0$, we also have that 
$$
{\rm vol}(\Delta(\Gamma_{\beta}(\mathcal J)))-\epsilon\le
\lim_{k\rightarrow \infty}\frac{\#(\Gamma_{\beta m}(\mathcal J(m)_k))}{k^dm^d}
=\frac{{\rm vol}(\Delta(\Gamma_{\beta m}(\mathcal J(m))))}{m^d}
\le{\rm vol}(\Delta(\Gamma_{\beta}(\mathcal J)))
$$
for $m\ge m_0$. Thus,
\begin{equation}\label{Leq8}
{\rm vol}(\Delta(\Gamma_{\beta}(\mathcal I)))
=\lim_{m\rightarrow \infty}\frac{{\rm vol}(\Delta(\Gamma_{\beta m}(\mathcal I(m))))}{m^d}
\end{equation}
and
\begin{equation}\label{Leq9}
{\rm vol}(\Delta(\Gamma_{\beta}(\mathcal J)))
=\lim_{m\rightarrow \infty}\frac{{\rm vol}(\Delta(\Gamma_{\beta m}(\mathcal J(m))))}{m^d}.
\end{equation}
So by (\ref{Leq6}),  (\ref{Leq8}), (\ref{Leq9}) and (\ref{Leq7})
$$
\begin{array}{lll}
F(\mathcal I,\mathcal J)&=& d!\displaystyle\lim_{n\rightarrow\infty}\frac{\ell_R(J_n/I_n)}{n^d}\\
&=&d![k':k']\left({\rm vol}(\Delta(\Gamma_{\beta}(\mathcal J)))
-{\rm vol}(\Delta(\Gamma_{\beta}(\mathcal I)))\right)\\
&=& d![k':k]\left(\displaystyle\lim_{m\rightarrow \infty}\frac{{\rm vol}(\Delta(\Gamma_{\beta m}(\mathcal J(m))))}{m^d}
-\displaystyle\lim_{m\rightarrow \infty}\frac{{\rm vol}(\Delta(\Gamma_{\beta m}(\mathcal I(m))))}{m^d}\right)\\
&=&\lim_{m\rightarrow\infty}\frac{d!}{m^d}\left(\lim_{k\rightarrow\infty}
\frac{\ell_R(J(m)_k/I(m)_k)}{k^d}\right)\\
&=&\lim_{m\rightarrow\infty}\frac{G(\mathcal J(m),\mathcal I(m))}{m^d},
\end{array}
$$
completing the proof of Theorem \ref{MultFormula} in the case that $R$ is a complete domain.

Now assume that $R$ is only analytically unramified. We will complete the proof of Theorem \ref{MultFormula} by proving the theorem in this case. We may continue to assume that $R$ is complete, and hence reduced. Let $P_1,\ldots , P_r$ be the minimal primes of $R$ and let $R_i = R/P_i$ for $1\leq i\leq r$. Let $T= \bigoplus_{i=1}^rR_i$ and let $\phi: R\to T$ be the natural inclusion. By Artin-Rees, there exists a positive integer $\lambda$ such that
\begin{equation*}
    \omega_n:= \phi^{-1}(m_R^nT)= R\cap m_R^nT\subset m_R^{n-\lambda}
\end{equation*}for $n\geq \lambda$. Thus, $m_R^n\subset \omega_n\subset m_R^{n-\lambda}$ for $n\geq \lambda$. In particular,
\begin{equation*}
        m_R^{mk}\subset \omega_{mk}\subset m_R^{mk-\lambda}
    \end{equation*} for $m,k\in\mathbb{Z}_{>0}$ such that $mk\geq \lambda$. We have that
    \begin{equation*}
        \omega_n = \phi^{-1}\big(m_R^nR_1\bigoplus\cdots \bigoplus m_R^nR_r\big)
        = (m_R^n+P_1)\cap \cdots\cap (m_R^n+P_r).
    \end{equation*} Let $\beta = (\lambda+1)c$. We have that $\omega_{\beta n}\subset m_R^{c(\lambda +1)n-\lambda}\subset m_R^{cn}$ for $n\geq 1$, so that $\omega_{\beta mk}\subset m_R^{cmk}$ for $m$, $n\geq 1$. Thus,
    \begin{equation*}
        \omega_{\beta m k}\cap I(m)_k = \omega_{\beta mk}\cap (m_R^{cmk}\cap I(m)_k)
        = \omega_{\beta mk}\cap (m_R^{cmk}\cap J(m)_k)
        =\omega_{\beta m k}\cap J(m)_k
    \end{equation*}for $m$, $k\geq 1$. In particular $\omega_{\beta n}\cap I_n = \omega_{\beta n} \cap J_n$ for $n\geq 1$.
Consequently,
    \begin{equation}\label{unramifiedeq1}
        \ell_R(J_n/I_n)
        = \ell_R(J_n/\omega_{\beta n}\cap J_n)
        -\ell_R(I_n/\omega_{\beta n}\cap I_n)
    \end{equation}for $n\geq 1$ and
    \begin{equation}\label{unramifiedeq2}
        \ell_R(J(m)_k/I(m)_k)
        = \ell_R(J(m)_k/\omega_{\beta mk}\cap J(m)_k)
        - \ell_R(I(m)_k/\omega_{\beta mk}\cap I(m)_k)
    \end{equation}for $m$, $k\geq 1$.

    Define $E_n^0 = R$, and
    \begin{equation*}
        E_n^j= (m_R^{\beta n}+P_1)\cap \cdots\cap (m_R^{\beta n}+P_j)
    \end{equation*}for $1\leq j\leq r$ and $n\geq 1$. Let    \begin{equation*}
        L_0^j= F_0^j = \beta_{m,0}^j=\alpha_{m,0}^j = R.
    \end{equation*}for $0\leq j\leq r$ and $m\geq 1$. Let $L_n^j = E_n^j\cap J_n$, $F_n^j = E_n^j\cap I_n$, $\beta_{m,k}^j = E_{mk}^j\cap J(m)_k$, and $\alpha_{m,k}^j = E_{mk}^j\cap I(m)_k$ for $0\leq j\leq r$ and $m,n,k\geq 1$. Note that $\beta_{m,k}^r = E_{mk}^r\cap J(m)_k = \omega_{\beta m k}\cap J(m)_k$, $\alpha_{m,k}^r = \omega_{\beta m k}\cap I(m)_k$, $L_n^r = \omega_{\beta n}\cap J_n$, and $F_n^r = \omega_{\beta n}\cap I_n$. We have isomorphisms
    \begin{equation}\label{epsmult3-1}
        L_n^j/L_n^{j+1}=L_n^j/\big(m_R^{\beta n}+P_{j+1}\big)\cap L_n^j
        \cong L_n^jR_{j+1}/(L_n^jR_{j+1})\cap m_{R_{j+1}}^{\beta n}.
    \end{equation}for $0\leq j\leq r-1$. Similarly,
    \begin{equation}
    \begin{split}\label{unramifiedeq3}
        F_n^j/F_n^{j+1}&\cong F_n^jR_{j+1}/(F_n^jR_{j+1})\cap m_{R_{j+1}}^{\beta n},\\
        \beta_{m,k}^j/\beta_{m,k}^{j+1}&\cong \beta_{m,k}^jR_{j+1}/((\beta_{m,k}^jR_{j+1})\cap m_{R_{j+1}}^{\beta mk},\\
        \alpha_{m,k}^j/\alpha_{m,k}^{j+1}&\cong \alpha_{m,k}^jR_{j+1}/((\alpha_{m,k}^jR_{j+1})\cap m_{R_{j+1}}^{\beta mk}.
    \end{split}
    \end{equation}
    We have the following equations:
    \begin{equation*}
    \begin{split}
    \ell_R(J_n/\omega_{\beta n}\cap J_n) &= \sum_{j=0}^{r-1}\ell_R(L_n^j/L_n^{j+1}),\\
        \ell_R(I_n/\omega_{\beta n}\cap I_n) &= \sum_{j=0}^{r-1}\ell_R(F_n^j/F_n^{j+1}),\\
        \ell_R(J(m)_k/ \omega_{\beta mk}\cap J(m)_k)
        &= \sum_{j=0}^{r-1}\ell_R(\beta_{m,k}^j/\beta_{m,k}^{j+1}),\\
        \ell_R(I(m)_k/ \omega_{\beta mk}\cap I(m)_k)
        &= \sum_{j=0}^{r-1}\ell_R(\alpha_{m,k}^j/\alpha_{m,k}^{j+1}).
    \end{split}
    \end{equation*}

    Then by (\ref{unramifiedeq1}) and (\ref{epsmult3-1}), we have that
    \begin{equation}\label{unramifiedeq4}
    \begin{split}
        &\ell_R\big(J_n/I_n\big) = \sum_{j=0}^{r-1}\Big( \ell_R\Big(L_n^j/L_n^{j+1}\Big) - \ell_R\Big(F_n^j/F_n^{j+1}\Big)\Big)
        \\&=\sum_{j=0}^{r-1}\Big(\ell_{R_{j+1}}\Big(L_n^jR_{j+1}\Big)/\Big(\Big(L_n^jR_{j+1}\Big)\cap m_{R_{j+1}}^{\beta n}\Big)
        - \ell_{R_{j+1}}\Big(F_n^jR_{j+1}/\Big(\Big(F_n^jR_{j+1}\Big)\cap m_{R_{j+1}}^{\beta n}\Big)\Big)
    \end{split}
    \end{equation}
    for $n\geq 1$, and (\ref{unramifiedeq2}) and (\ref{unramifiedeq3}) now yields that
    \begin{equation}\label{unramifiedeq5}
    \begin{split}
        &\ell_R\big(J(m)_k/I(m)_k\big) = \sum_{j=0}^{r-1}\Big( \ell_R\Big(\beta_{m,k}^j/\beta_{m,k}^{j+1}\Big) - \ell_R\Big(\alpha_{m,k}^j/\alpha_{m,k}^{j+1}\Big)\Big)\\
        &=\sum_{j=0}^{r-1}\Big(\ell_{R_{j+1}}\Big(\beta_{m,k}^jR_{j+1}/\Big(\Big(\beta_{m,k}^jR_{j+1}\Big)\cap m_{R_{j+1}}^{\beta m k}\Big)
        - \Big(\ell_{R_{j+1}}\Big(\alpha_{m,k}^jR_{j+1}/\Big(\Big(\alpha_{m,k}^jR_{j+1}\Big)\cap m_{R_{j+1}}^{\beta m k}\Big)\Big)
    \end{split}
    \end{equation}for $m$, $k\geq 1$.\\

    Fix $j\in\{0,\ldots,r-1\}$ and let $\overline{R}= R_{j+1}$. Take $\overline{J}_n = L_n^j\overline{R}$, $\overline{I}_n = \overline{J}_n\cap m_{\overline{R}}^{\beta n}$, $\overline{J}(m)_k = \beta_{m,k}^j\overline{R}$, and $\overline{I}(m)_k = \overline{J}(m)_k\cap m_{\overline{R}}^{\beta m k}$. Then $\overline{I}_n\subset \overline{J}_n$ and $\overline{I}(m)_k\subset \overline{J}(m)_k$ for all $n$, $m$, and $k$. Since $R$ is complete, we have that $\overline{R} = R/P_{j+1}$ is a complete domain.\\
    
    Fix $m\geq 1$. We will prove that the graded families $\overline{\mathcal{I}}(m):=\{\overline{I}(m)_k\}$,\\ $\overline{\mathcal{J}}(m):= \{\overline{J}(m)_k\},\textbf{ } \overline{\mathcal{I}}:=\{\overline{I}_n\}$, and $\overline{\mathcal{J}}:= \{\overline{J}_n\}$ in $\overline{R}$ satisfy the hypothesis of Theorem \ref{MultFormula}. Fixing $k\geq 1$, we have that
    \begin{equation*}
    \begin{split}
        \overline{J}(m)_k &= (E_{mk}^j\cap J(m)_k)\overline{R} \subset (E_{mk}^j\cap J_{mk})\overline{R} = L_{mk}^j\overline{R} = \overline{J}_{mk},\\
        \overline{I}(m)_k &= (E_{mk}^j\cap J(m)_k)\overline{R} \cap m_{\overline{R}}^{\beta mk}
        \subset (E_{mk}^j\cap J_{mk})\overline{R} \cap m_{\overline{R}}^{\beta mk}
        \\&=L_{mk}^j\overline{R} \cap m_{\overline{R}}^{\beta m k} = \overline{I}_{mk},\\
        \overline{J}(m)_1 &= (E_{m}^j\cap J(m)_1)\overline{R} = (E_{m}^j\cap J_m)\overline{R} = \overline{J}_m,\\
        \overline{I}(m)_1 &= (E_{m}^j\cap J(m)_1)\overline{R} \cap m_{\overline{R}}^{\beta m}
        = \overline{J}_m\cap m_{\overline{R}}^{\beta m} = \overline{I}_m.
    \end{split}
    \end{equation*}
    By definition, $\overline{J}(m)_k\cap m_{\overline{R}}^{\beta mk} = \overline{I}(m)_k$ for $m,k\geq 1$. Thus,
    \begin{equation}\label{unramifiedeq6}
    \begin{split}
        \overline{I}(m)_k\cap m_{\overline{R}}^{\beta mk} 
        =\overline{J}(m)_k\cap m_{\overline{R}}^{\beta mk}.
    \end{split}
    \end{equation}

We will now establish that 
    
        \begin{equation}\label{Neq1}
    \lim_{m\to\infty}\lim_{k\to\infty}
    \frac{\ell_{\overline{R}}\big(\overline{J}(m)_k/\overline{I}(m)_k\big)}{m^dk^d}
    =\lim_{n\to\infty}\frac{\ell_{\overline{R}}\big(\overline{J}_n/\overline{I}_n\big)}{n^d}.
    \end{equation}
     
         Fix $m\in \mathbb{Z}_{>0}$. Let $c'=\beta m$. By (\ref{unramifiedeq6}) we have that $\overline{I}(m)_k\cap m_{\overline{R}}^{c'k} = \overline{J}(m)_k\cap m_{\overline{R}}^{c'k}$ for all $k$. Since $\overline{J}(i)_1 = \overline{J}_i$ and $\overline{I}(i)_1=\overline{I}_i$ for all $i$, (\ref{unramifiedeq6}) also yields that $\overline{J}_i\cap m_{\overline{R}}^{\beta i} = \overline{I}_i\cap m_{\overline{R}}^{\beta i}$ for all $i$. Suppose that $\dim \overline{R}<d$. Then by Lemma \ref{DimlessthandLemma}, we have that $\lim_{n\to\infty}\ell_{\overline{R}}(\overline{J}_n/\overline{I}_n)/n^d=0$ and $\lim_{k\to\infty}\ell_{\overline{R}}(\overline{J}(m)_k/\overline{I}(m)_k)/k^d$ $=0$ for all $m$, and so,
    \begin{equation*}
        \lim_{m\to\infty}\lim_{k\to\infty}
    \frac{\ell_{\overline{R}}\big(\overline{J}(m)_k/\overline{I}(m)_k\big)}{m^dk^d}
    =0=\lim_{n\to\infty}\frac{\ell_{\overline{R}}\big(\overline{J}_n/\overline{I}_n\big)}{n^d}.
    \end{equation*}

    Now assume that $\dim \overline{R}=d$. We will show that if $\overline{I}_1=0$, then $\overline{I}_n=0$ for all $n$. Suppose that $\overline{I}_1=0$. Then $\overline{J}_1\neq \overline{R}$ (otherwise $0= \overline{J}_1\cap m_{\overline{R}}^{\beta}= m_{\overline{R}}^{\beta}$, while $\dim \overline{R} =d>0$). Consequently, $\overline{J}_1^{\beta}\subset \overline{J}_1\cap m_{\overline{R}}^{\beta}=0$, so that $(L_1^j+P_{j+1})/P_{j+1}= \overline{J}_1=0$. Hence, $m_R^{\beta}\cap J_1\subset E_1^j\cap J_1=L_1^j\subset P_{j+1}$. We have $J_1\neq R$ (otherwise $m_R^{\beta}\subset P_{j+1}$, so that $m_R = P_{j+1}$, while $\dim R>0$ and $P_{j+1}$ is minimal). Consequently, $I_1\subset J_1\subset P_{j+1}$. Then by Remark \ref{RemarkMinimal}, $J_n\subset P_{j+1}$. Thus, $L_n^j=E_n^j\cap J_n\subset P_{j+1}$, so that $\overline{I}_n\subset \overline{J}_n = (L_n^j+P_{j+1})/P_{j+1}=0$.\\

    By the case of Theorem \ref{MultFormula} when $R$ is a ($d$-dimensional) complete domain (replacing $R$, $c$, $I_n$, $J_n$, $I(m)_k$, and $J(m)_k$ with $\overline{R}$, $\beta$, $\overline{I}_n$, $\overline{J}_n$, $\overline{I}(m)_k$, and $\overline{J}(m)_k$ respectively), we have that
    \begin{equation*}
    \lim_{m\to\infty}\lim_{k\to\infty}
    \frac{\ell_{\overline{R}}\big(\overline{J}(m)_k/\overline{I}(m)_k\big)}{m^dk^d}
    =\lim_{n\to\infty}\frac{\ell_{\overline{R}}\big(\overline{J}_n/\overline{I}_n\big)}{n^d}
    \end{equation*}as desired, establishing (\ref{Neq1}).

    Equation (\ref{Neq1}) can be restated as
    \begin{equation}\label{unramifiedeq7}
    \begin{split}
        \lim_{m\to\infty}\lim_{k\to\infty}&\frac{\ell_R((\beta^j_{m,k}R_{j+1})/((\beta^j_{m,k}R_{j+1})\cap m_{R_{j+1}}^{\beta mk}))}{m^dk^d}
        \\&=\lim_{n\to\infty}\frac{\ell_R((L_n^jR_{j+1})/((L_n^jR_{j+1})\cap m_{R_{j+1}}^{\beta n}))}{n^d}
    \end{split}
    \end{equation}for $0\leq j\leq r-1$. Adding the equations (\ref{unramifiedeq7}) over $j = 0,\ldots,r-1$ shows that
    \begin{equation}\label{unramifiedeq8}
    \begin{split}
        \lim_{m\to\infty}\lim_{k\to\infty}&
    \frac{\sum_{j=0}^{r-1}\ell_R((\beta_{m,k}^jR_{j+1})/(\beta_{m,k}^jR_{j+1})\cap m_{R_{j+1}}^{\beta m k}))}{m^dk^d}
    \\&=\lim_{n\to\infty}\frac{\sum_{j=0}^{r-1}\ell_R((L_n^jR_{j+1})/(L_n^jR_{j+1}\cap m_{R_{j+1}}^{\beta n}))}{n^d}.
    \end{split}
    \end{equation}Applying the above argument proving equation (\ref{unramifiedeq8}) to $\overline{J}_n = F_n^j\overline{R}$, $\overline{I}_n = \overline{J}_n\cap m_{\overline{R}}^{\beta n}$, $\overline{J}(m)_k = \alpha_{m,k}^j\overline{R}$ and $\overline{I}(m)_k=\overline{J}(m)_k\cap m_{\overline{R}}^{\beta m k}$, we obtain
    \begin{equation}\label{unramifiedeq9}
    \begin{split}
        \lim_{m\to\infty}\lim_{k\to\infty}&
    \frac{\sum_{j=0}^{r-1}\ell_R((\alpha_{m,k}^jR_{j+1})/(\alpha_{m,k}^jR_{j+1})\cap m_{R_{j+1}}^{\beta m k}))}{m^dk^d}
    \\&=\lim_{n\to\infty}\frac{\sum_{j=0}^{r-1}\ell_R((F_n^jR_{j+1})/(F_n^jR_{j+1}\cap m_{R_{j+1}}^{\beta n}))}{n^d}.
    \end{split}
    \end{equation}Subtracting (\ref{unramifiedeq9}) from (\ref{unramifiedeq8}), and then applying (\ref{unramifiedeq4}) and (\ref{unramifiedeq5}) gives that

\begin{equation}\label{unramifiedeq10}
    \lim_{m\to\infty}\lim_{k\to\infty}\frac{\ell_R(J(m)_k/I(m)_k)}{m^dk^d}
    =\lim_{n\to\infty}\frac{\ell_R(J_n/I_n)}{n^d}.
\end{equation}Multiplying both sides of (\ref{unramifiedeq10}) by d! gives that
\begin{equation*}
    \lim_{m\to\infty}\frac{G(\mathcal{I}(m),\mathcal{J}(m))}{m^d}=\mathcal{F}(\mathcal{I},\mathcal{J}).
\end{equation*}This completes the proof of Theorem \ref{MultFormula}.

\section{Epsilon multiplicity as a limit of Amao multiplicities}\label{SecEM}

In this section, we deduce Theorem \ref{TheoremA} from Theorem \ref{MultFormula}.

\begin{Lemma}\label{EpsTheorem3} Let $R$ be any Noetherian local ring. For $i\in\mathbb{Z}_{>0}$, and $I\subset R$ any ideal, we have
    \begin{equation*}
        (I^{\text{sat}})^i\subset (I^i)^{\text{sat}}.
    \end{equation*}
\end{Lemma}
\begin{proof}
    Fix $i\geq 1$. There exists $t\in\mathbb{Z}_{>0}$ such that for $n\geq t$, we have
    \begin{equation*}
    \begin{split}
        (I^i)^{\text{sat}} &= I^i:m_R^n\\
        I^{\text{sat}} &= I:m_R^n.
    \end{split}
    \end{equation*}Consequently,
    \begin{equation*}
        (I^{\text{sat}})^i = (I:m_R^t)^i\subset
        I^i:m_R^{it} = I^i:m_R^t = (I^i)^{\text{sat}}.
    \end{equation*}
\end{proof}

\begin{Lemma}\label{EpsTheorem4}
    Let $R$ be a Noetherian local ring and $I\subset R$ an ideal. Then there exists $c\in\mathbb{Z}_{>0}$ such that for $m,k\geq 1$, we have
    \begin{equation*}
        I^{mk}\cap m_R^{cmk}
        = [(I^m)^{\text{sat}}]^k\cap m_R^{cmk}.
    \end{equation*}
\end{Lemma}
\begin{proof}
By Theorem 3.4 \cite{S1}, there exists $c\in \mathbb{Z}_{>0}$ such that each power $I^n$ of $I$ has an irredundant primary decomposition
\begin{equation*}
    I^n = q_1(n)\cap \cdots \cap q_t(n)
\end{equation*} where $q_1(n)$ is $m_R-$primary and $m_R^{nc}\subset q_1(n)$ for all $n$. Then 
$$
(I^n)^{\text{sat}}=q_2(n)\cap\cdots\cap q_t(n),
$$
and consequently
\begin{equation*}
    I^n\cap m_R^{cn} = m_R^{cn}\cap q_2(n)\cap\cdots\cap q_t(n) = m_R^{cn}\cap (I^n)^{\text{sat}}.
\end{equation*}
Then for $m,k\in\mathbb{Z}_{>0}$,
    \begin{equation*}
        I^{mk}\cap m_R^{cmk} = (I^{mk})^{\text{sat}}\cap m_R^{cmk}.
    \end{equation*}By Lemma \ref{EpsTheorem3}, we now have
    \begin{equation*}
        I^{mk}\cap m_R^{cmk}
        \subset [(I^m)^{\text{sat}}]^k\cap m_R^{cmk}
        \subset (I^{mk})^{\text{sat}}\cap m_R^{cmk}
        = I^{mk}\cap m_R^{cmk}
    \end{equation*}for $m,k\geq 1$. This completes the proof of Lemma \ref{EpsTheorem4}.
\end{proof}

We are now ready to prove Theorem \ref{TheoremA}. The case when $d=0$ is immediate since an analytically unramified local ring of dimension zero is a field.\\

Now we prove the theorem when $d>0$. Define ideals of $R$ for $n,m,k\geq 1$
    \begin{equation*}
    \begin{split}
        J_n&:=(I^n)^{\text{sat}},\\
        I_n&:= I^n\subset J_n,\\
        J(m)_k&:= [(I^m)^{\text{sat}}]^k,\\
        I(m)_k&:= I^{mk} = I_{mk}.
    \end{split}
    \end{equation*}We have $I(m)_k=I^{mk}\subset [(I^m)^{\text{sat}}]^k=J(m)_k$. Additionally, $J(m)_k = [(I^m)^{\text{sat}}]^k\subset (I^{mk})^{\text{sat}}=J_{mk}$ by Lemma \ref{EpsTheorem3}. By the definitions, $J(m)_1 = J_m$ and $I(m)_1 = I_m$ for $m\geq 1$. For $m\geq 1$, $\mathcal{I}:=\{I_n\}$, $\mathcal J=\{J_m\}$,
    $\mathcal{I}(m):= \{I(m)_k\}$, and $\mathcal{J}(m):=\{J(m)_k\}$ are all graded families of ideals in $R$.  By Lemma \ref{EpsTheorem4}, we may find $c\in\mathbb{Z}_{>0}$ such that
    \begin{equation*}
        I(m)_k\cap m_R^{cmk} = J(m)_k\cap m_R^{cmk}
    \end{equation*}for $m,k\geq 1$. Suppose that $P$ is a minimal prime of $R$ such that $I_1\subset P$. Then for $i\geq 1$, $I_i=I^i\subset I= I_1\subset P$. Hence the graded families $\mathcal{I}$, $\mathcal{J}$, $\mathcal{I}(m)$, and $ \mathcal{J}(m)$ satisfy the hypothesis of Theorem \ref{MultFormula}. Thus the conclusions of Theorem \ref{TheoremA}  hold by Theorem \ref{MultFormula}.

\end{document}